\documentclass[12pt]{article}
\usepackage{amssymb}
\newcommand{\qed}{\hfill \rule{2.5mm}{2.5mm}}
\newcommand{\R}{{\mathbb R}}
\newcommand{\C}{{\mathbb C}}
\newcommand{\tvect}[2]{\ensuremath{\negthinspace\begin{pmatrix}#1 \\ #2 \end{pmatrix}}}
\newcommand{\N}{{\mathbb N}}
\newcommand{\Z}{{\mathbb Z}}
\newcommand{\sgn}{{\rm sgn}\,}

\renewcommand{\tvect}[2]{\left(\begin{array}{cc}{ #1}\\ { #2}\end{array}\right)}

\begin{document}
\newtheorem{thm}{Theorem}[section]
\newtheorem{defs}[thm]{Definition}
\newtheorem{lem}[thm]{Lemma}
\newtheorem{example}[thm]{Example}
\newtheorem{cor}[thm]{Corollary}
\newtheorem{prop}[thm]{Proposition}
\renewcommand{\theequation}{\arabic{section}.\arabic{equation}}
\newcommand{\newsection}[1]{\setcounter{equation}{0} \section{#1}}
%%%%%%%%%%%%% title %%%%%%%%%%%%%%%%%%%%%%%%%%%%%%%
\title{Borg's Periodicity Theorems\\ for first order self-adjoint systems\\ with complex potentials
      }
%%%%%%%%%%%%%%%%%%%%%%%%%%%%%%%%%%%%%%%%%%%%%%
\author{
Sonja Currie\footnote{ Supported by NRF grant number IFR2011040100017}, Thomas T. Roth,
Bruce A. Watson \footnote{Supported by the Centre for Applicable Analysis and
Number Theory, the DST-NRF Centre of Excellence in Mathematical and Statistical Sciences 
and by NRF grant number IFR2011032400120.} \\
School of Mathematics\\
University of the Witwatersrand\\
Private Bag 3, P O WITS 2050, South Africa }
\maketitle
%%%%%%%%%%%%%%%%%%%% abstract %%%%%%%%%%%%%%%%%%%%%%
\abstract{
A self-adjoint first order system with Hermitian $\pi$-periodic potential $Q(z)$, integrable on compact sets, is considered. 
 It is shown that all zeros of $\Delta + 2e^{-i\int_0^\pi \Im q dt}$ are double zeros if and only
if this self-adjoint system is unitarily equivalent to one in which $Q(z)$ is $\frac{\pi}{2}$-periodic. Furthermore, the zeros 
of $\Delta - 2e^{-i\int_0^\pi \Im q dt}$ are all double zeros
  if and only if the associated self-adjoint system is unitarily equivalent to one in which $Q(z) = \sigma_2 Q(z) \sigma_2$.
 Here $\Delta$ denotes the discriminant of the system and $\sigma_0$, $\sigma_2$ are Pauli matrices. Finally, it is shown
 that all instability intervals vanish if and only if $Q = r\sigma_0 + q\sigma_2$, for some real valued $\pi$-periodic functions $r$ and $q$ 
integrable on compact sets.
}
%%%%%%%%%%%%%%%%%%%%%%%%%%%%%%%%%%%%%%%%%%%%%%
\parindent=0in
\parskip=.2in
%%%%%%%%%%%%%%% introduction %%%%%%%%%%%%%%%%%%%%%%%%%%
\newsection{Introduction\label{sec-intro}}

Self-adjoint systems have been studied extensively in the last century, see \cite{LaS}-\cite{LS2}. Periodic problems for
 self-adjoint systems with integrable 
potentials have received consistent attention, \cite{EBS}. This is especially true recently for the Ambarzumyan and Borg uniqueness-type results,
 \cite{GKM}-\cite{fGaKkM}, \cite{mK} and \cite{cYxY}.
 It should be noted that these results pertain mainly to regular and singular inverse problems with $2n\times 2n$ potentials with matrix
 valued
 entries. These classes of problems are not as developed as inverse problems for canonical $2 \times 2$ systems, in which many inverse results pertaining to
 uniqueness have been investigated, \cite{mGtD1}-\cite{bWat}. Never the less, $2 \times 2$ self-adjoint systems are an active area of study in physics communities in
 which they are referred to as 
the Ablowitz-Kaup-Newell-Segur equation, \cite{lA}-\cite{fGKM}, \cite{bWat}, and the Zakarov-Shabat equation, \cite{pBmV}-\cite{vsGgV}. This alludes to a link
 between self-adjoint systems and completely integrable systems which is being actively investigated, \cite{nAyK}-\cite{baD}. 

The results in this work where first proved for the Sturm-Liouville eigenvalue problem by Ambarzumyan, \cite{ABar}, Borg, \cite{BG}, and later Hochstadt,
 \cite{HH1,HH2}.
In particular, in Borg's paper he proved an existence result for periodic potentials which has largely gone unstudied for the self-adjoint system. 
 Self-adjoint systems with absolutely continuous potentials are reducible to Sturm-Liouville equations.
 This is not possible in general for self-adjoint systems with potentials
integrable on compact sets. These systems present challenges that make the results
 of this work a non-trivial extension of the aforementioned works. These challenges are: \\
a) Existing asymptotics for self-adjoint systems do not allow the generality of potential considered here. Difficulties in deriving such asymptotics have been
 discussed in the remark of \cite[pp. 1464]{fS1}. \\
b) Self-adjoint systems are spectrally identical to those obtained by certain gauge transformations, thus uniqueness results are
 not possible in general. These transformations have been investigated in \cite{LaS} and \cite{fGsC3}. \\ \\
In Section \ref{asymp1}, resolution of the first term in the solution asymptotics for all values of the eigenparameter in $\C$ are established for
 Hermitian
 $Q$ integrable on compact sets. The authors are only aware of solution asymptotics on open sectors in $\C$ for canonical systems with potentials integrable on
 compact sets and systems with absolutely continuous potentials, see
 \cite[pp. 191]{LaS}, \cite{fS1}, \cite[pp. 3492]{fGsC3}. In Section \ref{Chardeter1} we introduce the $\sigma_i$-determinants.
 Lemmas \ref{CanonSystI} and \ref{CanonSystI2} establish an important relation between the $\mathbb{I}$-discriminant
 of a self-adjoint system and the behaviour of the fundamental solution at $\pi$ and $\frac{\pi}{2}$ (these are also referred to as monodromy matrices). These Lemmas
 are essential for studying the inverse problem.

The main results of this work, Theorems \ref{perthm} and \ref{antiperthm} are the self-adjoint system
 analogues to the Sturm-Liouville results
obtained in \cite{BG, HH2} and \cite{HH3}, respectively. Corollary
 \ref{BorgUnique} shows that uniqueness is possible only in the case when $Q$ is in canonical form. Finally, it is shown as a pleasant consequence, that
Borg's uniqueness result for canonical systems is derivable from the Borg Periodicity Theorems. Furthermore, the extent to which this uniqueness result 
fails for self-adjoint systems is characterised. 
This work uses ideas presented in \cite{HH2}, however, as far as the authors are aware, the results presented here are new.  

%%%%%%%%%%% Preliminaries %%%%%%%%%%%%%%%%%%%%%%%%%%%%%
\newsection{Preliminaries\label{sec-prelim}}

Let
\begin{eqnarray}\label{DifferentialExpr}
	\ell Y := JY' + QY,
\end{eqnarray}
and consider the differential equation
\begin{equation} \label{eValProbFormal}
	\ell Y = \lambda Y
\end{equation}
where 
\begin{eqnarray}\label{QDescription}
	J = \left( \begin{array}{cc}
0 &  1\\
-1 & 0 \end{array} \right) \quad \mbox{ and } \quad Q = \left( \begin{array}{cc}q_1 &  q\\q^* & q_2 \end{array} \right),
\end{eqnarray}
$q$ is complex valued, $q_1$ and $q_2$ are real and $Q$ is $\pi$-periodic and integrable on $[0,\pi)$.
Let $Y_i = \tvect{y_{i1}(z)}{y_{i2}(z)}, i=1,2,$ be solutions of (\ref{eValProbFormal}) with initial values 
given by
\begin{eqnarray}\label{initialConds}
	 [ Y_1(0) \mbox{ } Y_2(0) ] = \mathbb{I},
\end{eqnarray}
where $\mathbb{I}$ is the $2\times 2$ identity matrix.  Set $\mathbb{Y}=[Y_1 \mbox{ } Y_2]$. We recall that the Pauli matrices are given by
 $\sigma_0=\mathbb{I}$, 
\begin{equation}
	\sigma_1 =	\left( \begin{array}{cc}
	0 &  1\\
	1 & 0 \end{array} \right), \quad
\sigma_2 =	\left( \begin{array}{cc}
0 &  i\\
-i & 0 \end{array} \right), \quad
\sigma_3 =	\left( \begin{array}{cc}
1 &  0\\
0 & -1 \end{array} \right).
\end{equation}
Here $\sigma_2 = i J$. The set of Pauli matrices form a basis for $M_2(\C)$, the $2\times 2$ matrices over $\C$,
and 
\begin{eqnarray*}
	\sigma_i\sigma_j = i\epsilon_{ijk}\sigma_k + \delta_{ij}\mathbb{I}, \quad \mbox{ for } \quad i,j = 1,2,3, k \neq i,j,
\end{eqnarray*}
where $\epsilon_{ijk}$ and $\delta_{ij}$ are the Levi-Civita permutation and Kronecker delta symbols, respectively. Note that $\epsilon_{ijk}$ is 
$1$ if the number of permutations of $(i,j,k)$ into $(1,2,3)$ is even, $-1$ if the number of permutations of $(i,j,k)$ into $(1,2,3)$ is odd, and 
zero if any of the indices are repeated.
 Furthermore the $2\times 2$ matrices over $\C, M_2(\C),$
form an inner product space with inner product defined by
\begin{equation}
	\langle H, F \rangle_{Lin} = Tr\{ H^T\overline{F}\}, \quad \mbox{ for }H,F.
\end{equation}
 For any $H =\sum_{i=0}^3a_i\sigma_i \in M_2(\C)$, the determinant is 
given by
\begin{equation}\label{detDefn}
\mbox{det}({H}) = a_0^2 - a_1^2 - a_2^2 - a_3^2.	
\end{equation} 
Define the {\it $\sigma_i$-symmetric} and {\it $\sigma_i$-skewsymmetric} subspaces $S_+^{\sigma_i}$ and $S_-^{\sigma_i}$ of
 $M_2(\C)$
 as
\begin{equation}
	S_+^{\sigma_i} = \{ x \in GL(2,\C) : x\sigma_i = \sigma_ix \} \quad \mbox{and} \quad S_-^{\sigma_i} = \{ x \in M_2(\C) :
	 x\sigma_i = -\sigma_i
	x \}.
\end{equation}
We have the product space $M_2(\C) = S_-^{\sigma_i} \oplus S_+^{\sigma_i}$.

The {\it J-decomposition} of $Q$ in $M_2(\C) = S_-^{J} \oplus S_+^{J}$ is
\begin{equation}
	Q = Q_1 + Q_2,
\end{equation}
where 
\begin{equation}\label{Q1andQ2Defn}
	\quad Q_1 = \Re (q)\sigma_1 +
	 \textstyle\frac{1}{2}(q_1 -
	 q_2)\sigma_3 \quad \mbox{and} \quad Q_2 = \textstyle\frac{1}{2}(q_1 + q_2)\sigma_0 + \Im (q) \sigma_2 .
\end{equation}
We see that $Q_1$ and $Q_2$ are the projections of $Q$ onto $S_-^{J}$ and $S_+^{J}$, respectively. We note for later that $Q_2J = JQ_2$ and $JQ_2$ and 
$\int_0^zJQ_2$ commute.  A potential $Q$ is said to be in canonical form if $Q_2 = 0$, 
that is, $Q = Q_1$.

Since $\mathbb{Y}$ is a fundamental system for
 (\ref{eValProbFormal}), $\mathbb{Y}\in GL(2,\C)$. Furthermore setting
\begin{equation}\label{DiscrimDef01}
	\begin{array}{cc}
	\Delta^\mathbb{I} = y_{11}(\pi) + y_{22}(\pi), &  \nabla^\mathbb{I} = y_{11}(\pi) - y_{22}(\pi),\\
	\Delta^J = y_{21}(\pi) - y_{12}(\pi), & \nabla^J = y_{21}(\pi) + y_{12}(\pi), \end{array}
\end{equation}
$\mathbb{Y}(\pi)$ may be represented as
\begin{equation} \label{discrimRepY}
	\mathbb{Y}(\pi) = \frac{1}{2}\left( \begin{array}{cc}
	\Delta^\mathbb{I} + \nabla^\mathbb{I}  &  \Delta^J + \nabla^J\\
	\nabla^J - \Delta^J & \Delta^\mathbb{I} - \nabla^\mathbb{I} \end{array} \right).
\end{equation}
Thus expressed in terms of the Pauli basis for $GL(2,\C)$ we have
\begin{equation}\label{GLYLin}
	 \mathbb{Y}(\pi) = \frac{1}{2}(\Delta^\mathbb{I}\mathbb{I} + \Delta^JJ +
	 \nabla^\mathbb{I}\sigma_3 + \nabla^J\sigma_1).
\end{equation}
A direct computation using (\ref{detDefn}) and (\ref{GLYLin}) with $\mbox{det}(\mathbb{Y}) = 1$ gives
\begin{eqnarray}\label{WronGL}
	(\Delta^\mathbb{I})^2 + (\Delta^J)^2 - (\nabla^\mathbb{I})^2 - (\nabla^J)^2 = 4.
\end{eqnarray}
Similar relations to (\ref{DiscrimDef01})-(\ref{WronGL}) for $\mathbb{Y}(\frac{\pi}{2})$ and $\mathbb{Y}(-\frac{\pi}{2})$
may be obtained, the symbols contained in these relations are denoted by the subscript $+$ and $-$, respectively. 
Let $\mathbb{H} = \mathcal{L}_2(0,\pi)\times\mathcal{L}_2(0,\pi)$ be the Hilbert space with inner product
\begin{eqnarray*}
\langle Y,Z\rangle = \int_0^\pi Y(t)^T \overline{Z}(t) dt \quad \mbox{ for } Y,Z \in \mathbb{H},
\end{eqnarray*}
and norm $\| Y \|^2_2 := \langle Y,Y \rangle$. The Wronskian of $Y,Z\in\mathbb{H}$ is $ \mbox{Wron}[Y,Z] = Y^TJZ$.
We consider the following operator eigenvalue problems
\begin{eqnarray} \label{eValProb}
	L_iY = \lambda Y, \qquad  \qquad i = 1,...,4,
\end{eqnarray}
where $L_i = \ell|_{\mathcal{D}(L_i)}$ with
\begin{eqnarray}
   \mathcal{D}(L_i)=\left\{ Y=\tvect{y_1}{y_2} \,:\,y_1,y_2 \in \mbox{AC}, \ell Y \in\mathbb{H},\mbox{Y obeys } (BC_i) \right\}.
\end{eqnarray}
Here conditions $(BC_i)$ are  \\
\begin{eqnarray}
	Y(0) =&Y(\pi),\qquad   &(BC_1), \label{per}\\  % BC_1
	Y(0) =&-Y(\pi),\qquad   &(BC_2), \label{antper}\\ % BC_2
	y_1(0) =&y_1(\pi)=0,\qquad  &(BC_3), \label{1Dir}\\
	y_2(0) =&y_2(\pi)=0,\qquad &(BC_4). \label{2Dir} 
\end{eqnarray}

%%%%%%%%%%%%%% Solution Asymptotics %%%%%%%%%%%%%%%%%%%
\newsection{Solution Asymptotics}\label{asymp1}
We now give an asymptotic approximation for $\mathbb{Y}$ in the case of $|\lambda|$ large. 
We will make use of the following operator matrix norm $$|[c_{ij}]|=\max_{j}{\sum_i |c_{ij}|}.$$

\begin{lem}\label{LemAsymptotics}
	Let $Q = Q_1 + Q_2$ (as in (\ref{Q1andQ2Defn})) be complex valued and integrable on $[0,\pi)$.
	 The matrix solutions $\mathbb{Y}$ and $\mathbb{U}$ of $\ell Y = \lambda Y$ satisfying the 
	 conditions, $\mathbb{Y}(0) = \mathbb{I}=\mathbb{U}(\pi)$, are of order $1$. For $z \in \R$ and $\lambda = re^{i\theta}$ with $r \rightarrow \infty$, we 
	have uniformly in $\theta$ and $z$, that
	\begin{equation}\label{Asymptot1.1}
		\mathbb{Y}(z) = e^{-J\lambda z}e^{J\int_0^z Q_2 dt} + o(e^{|\Im \lambda z|}),
	\end{equation}
	\begin{equation}\label{Asymptot1.2}
		\mathbb{U}(z) = e^{-J\lambda (\pi - z)}e^{J\int_0^{(\pi - z)} Q_2 dt} + o(e^{|\Im \lambda(\pi - z)|}).
	\end{equation}
\end{lem}

\begin{proof}
	Consider the transformation $\mathbb{Y}(z) = e^{J\int_0^z Q_2 dt}\tilde\mathbb{Y}(z)$ for $z \geq 0$.
	Substituting this transformation into (\ref{DifferentialExpr}) gives 
	\begin{equation}\label{tildeDiracEquation}
		J\tilde\mathbb{Y}' + \tilde{Q}\tilde\mathbb{Y} = \lambda \tilde\mathbb{Y}
	\end{equation}
	where 
	\begin{eqnarray}
		\tilde{Q}(z)  &=& e^{-J\int_0^z Q_2 dt}Q_1(z)e^{J\int_0^z Q_2 dt}.
	\end{eqnarray}
		
	Notice that $\tilde{Q}$ is a real canonical 
	matrix. Let $\tau := \Im\lambda$ and $\rho := \Re\lambda$.  Using variation of parameters, \cite[pp. 74]{CodaLev}, $\tilde\mathbb{Y}$ obeys
		 the integral equation
	\begin{equation}
		\tilde\mathbb{Y}(z) = e^{-\lambda J z} + 
		\int_0^z e^{-\lambda J(z - t)}J\tilde{Q}(t) \tilde\mathbb{Y}(t) dt.
	\end{equation}
	Setting $\tilde\mathbb{Y}(z) = e^{ \tau z}\mathbb{V}(z)$ we have
	\begin{equation}\label{VIntEqn}
		\mathbb{V}(z) = e^{-(\lambda J  + \mathbb{I}\tau )z} + \int_0^z e^{-(\lambda J + \mathbb{I}\tau)
		 (z-t)}J\tilde{Q}(t)\mathbb{V}(t)dt,
	\end{equation}
	giving
	\begin{equation}
		|\mathbb{V}(z)| \leq 1 + \int_0^z|\tilde{Q}||\mathbb{V}|dt.
	\end{equation}
	Using Gronwall's inequality, \cite[Lemma 6.3.6]{hormander}, we have the estimate $\mathbb{V} = O(1)$, thus  $\tilde\mathbb{Y} =
	 O(e^{\tau z})$.
Set $W(z) := e^{-(J\lambda + \mathbb{I}\tau)z}$. Substituting (\ref{VIntEqn}) back into itself gives
\begin{equation}\label{BlockIntEst2}
	\mathbb{V}(z) = W(z) + \int_0^z JW(z-t)\tilde{Q}(t)W(t)dt + 
	\int_0^z\int_0^t W(z-t)\tilde{Q}(t)W(t-s)\tilde{Q}(s) \mathbb{V}(s)ds dt,
\end{equation}
since $\tilde{Q}J = -J\tilde{Q}$. For $x,y \in \R$, $z \geq 0$, we have
\begin{eqnarray}
	W(x)\tilde{Q}(z)W(y)   &=& e^{-(\lambda J + \mathbb{I}|\tau|)x}e^{(\lambda J - \mathbb{I}|\tau|)y}\tilde{Q}(z),\\
						&=& e^{-\rho J(x-y)}e^{-|\tau|(x+y)}e^{-i\tau J(x-y)}\tilde{Q}(z),\label{WQWasymp2}
\end{eqnarray}
Furthermore, setting $f(x,y) := e^{-|\tau|(x+y)}e^{-i\tau J(x-y)}$ we have
\begin{equation}\label{BlockIntEst232}
   f(x,y) =\frac{1}{2}\mathbb{I}(e^{-2|\tau| x} + e^{-2|\tau| y}) + \frac{\sgn\tau}{2i}J
	(e^{-2|\tau| x}- e^{-2|\tau| y}),
\end{equation}
thus combining (\ref{WQWasymp2}) and (\ref{BlockIntEst232}) gives
\begin{equation}
	W(x)\tilde{Q}W(y) = O(|\tilde{Q}|e^{-2|\tau| \min \{x,y \} }).  \label{WQWasymp1}
\end{equation}
 From (\ref{WQWasymp1}) we have the following bound  
\begin{equation}\label{bound1LDCT}
	|W(z-t)\tilde{Q}(t)W(t)| \leq k|\tilde{Q}(t)|e^{-2\min\{z-t, t\}}, 
\end{equation}
 for some $k> 0$, independent of $\lambda$, $x$ and $y$. Using (\ref{bound1LDCT}), the Lebesgue dominated convergence theorem shows that 
\begin{equation}
	\int_0^z W(z-t)\tilde{Q}(t)W(t) dt = O\left( \int_0^z |\tilde{Q}(t)|e^{-2|\tau|\min\{z-t,t\})} dt \right),
\end{equation}
tends to zero as $|\tau|$ tends to infinity. While for $|\tau| = c < c'$, using (\ref{WQWasymp2}), we have that the second term on the right hand side of 
(\ref{BlockIntEst2}) is equal to
\begin{equation}\label{BlockIntEst23}
	\int_0^z (\mathbb{I}\cos\sigma(z-2t) -J\sin\sigma(z-2t))f(z-t,t)\tilde{Q}(t)dt,
\end{equation}
where $f(z-t,t)\tilde{Q}(t)$ is integrable on $[0,\pi]$. Thus by the Riemann-Lebesque Lemma, (\ref{BlockIntEst23}) tends to zero as $|\rho|$ tends to infinity.
Hence the second term on the right hand side of (\ref{BlockIntEst2}) tends to zero uniformly in $\arg(\lambda)$ as $|\lambda|$ tends to infinity. The uniformity
here follows from the uniformity of this limit as $|\tau|$ tends to infinity, thus this limit holds as $|\sigma|$ tends to infinity for fixed $c$.

By changing the order of integration, the double integral in (\ref{BlockIntEst2}) is equal to
\begin{equation}\label{3rdTermRHS}
	\int_0^z\left(\int_\tau^z W(z- t)\tilde{Q}(t)W(t - \tau) dt \right)\tilde{Q}(\tau)\mathbb{V}(\tau)d\tau.
\end{equation}
From the reasoning above, the inner integral in (\ref{3rdTermRHS})
tends to zero as $|\lambda|$ tends to infinity, thus, as $\mathbb{V}$ is bounded, so does the double integral. 
 So from (\ref{BlockIntEst2}) for large $|\lambda|$,
		\begin{equation}\label{VLittleOEst}
			\mathbb{V}(z) = e^{-(\lambda J  + \mathbb{I}|\tau|)z} + o(1).
		\end{equation}   
		Substituting (\ref{VLittleOEst}) back into the expression for $\tilde\mathbb{Y}$ gives
	\begin{equation}\label{PositiveAsymptot1}
		\mathbb{Y}(z) = e^{-J\lambda z}e^{J\int_0^z Q_2 dt} + o(e^{|\tau|z}) \quad \mbox{ for }z \geq 0.
	\end{equation}
	 Assuming that $z \leq 0$, we may apply the transformation $\hat{z} = -z$, $\hat{Y}(\hat{z})  = Y(z)$,
	$\hat{Q}(\hat{z}) = -Q(z)$	and $\hat\lambda = -\lambda$	to transform $\ell Y = \lambda Y$ into 
	\begin{equation}
		J\hat{Y}'(\hat{z}) + \hat{Q}(\hat{z})\hat{Y}(\hat{z}) = \lambda\hat{Y}(\hat{z}).
	\end{equation}
	From the above work $\hat\mathbb{Y}(\hat{z})$ is given by
	\begin{equation}
		\hat\mathbb{Y}(\hat{z}) = e^{-J\hat\lambda \hat{z}}e^{J\int_0^{\hat{z}} \hat{Q}_2 dt} + o(e^{|\tau|\hat{z}}).
	\end{equation}
	Thus substituting the transformations above we have
	\begin{equation}\label{NegativeAsymptot1}
		\mathbb{Y}(z) = e^{-J\lambda z}e^{J\int_0^{z} Q_2 dt} + o(e^{-|\tau|z}) \quad \mbox{ for }z \leq 0.
	\end{equation}
	Combining (\ref{PositiveAsymptot1}) and (\ref{NegativeAsymptot1}) gives (\ref{Asymptot1.1}).
	To obtain (\ref{Asymptot1.2}), set $\check\mathbb{U}(\check{x}) := \mathbb{U}(\pi - x)$ where $\check{x} = \pi - x$. Thus
	 $\check\mathbb{U}$
	 with 
	$\check\mathbb{U}(0) = \mathbb{I}$ is a solution to 
	$\ell Y = \lambda Y$ with potential $\check{Q}(\check{x}) := -Q(\pi - x)$. Finally we can apply (\ref{Asymptot1.1}) to obtain
	 (\ref{Asymptot1.2}). \qed
\end{proof}

%%%%%%%%%%% characteristic determinant %%%%%%%%%%%%%%%%%%%%%%%%%%%%%
\newsection{The Characteristic Determinant}\label{Chardeter1}
Consider the problem of
\begin{eqnarray}\label{floquetEqn}
	Y(z + \pi) = \rho(\lambda) Y(z),\quad\mbox{ for all }\quad z\in \R,
\end{eqnarray}
where $Y$ is a non-trivial solution of (\ref{eValProbFormal}) with $Q = Q_1$, and $\rho(\lambda)\in \C$. Here $\rho(\lambda)$ is multivalued and $Y$ can
 be
represented as $Y(z) = \mathbb{Y}(z)\underline{v}$, for some $\underline{v} \in \R^2\setminus\{0\}$. Since $Q$ is $\pi$-periodic,
we have $\mathbb{Y}(z + \pi) = \mathbb{Y}(z)\mathbb{Y}(\pi)$, which together with (\ref{floquetEqn}) for $z = 0$ yields
\begin{equation}\label{floquetEqn2}
	(\mathbb{Y}(\pi) - \rho\mathbb{I})\underline{v} = 0.
\end{equation}
A necessary and sufficient condition for the existence of nontrivial solutions of (\ref{floquetEqn2}) is 
$\mbox{det}(\mathbb{Y}(\pi) - \rho\mathbb{I}) = 0$. This may be expressed, via (\ref{GLYLin}), as
\begin{equation}\label{DetFloEqn}
\mbox{det}((\Delta^\mathbb{I} - 2\rho)\mathbb{I}	+ \Delta^JJ +
	 \nabla^\mathbb{I}\sigma_3 + \nabla^J\sigma_1) = 0.
\end{equation}
Using (\ref{detDefn}) and (\ref{WronGL}) to simplify (\ref{DetFloEqn}), we obtain
\begin{eqnarray}\label{DiracCharacteristic}
	\rho^2 - \rho\Delta^{\mathbb{I}} + 1 =  0.
\end{eqnarray}
The quantity $\Delta^\mathbb{I}$ will be called the {\it $\mathbb{I}$-discriminant} of the problem (\ref{eValProbFormal}) on
 $[0,\pi)$, and the solutions  $\rho = \frac{\Delta^\mathbb{I} \pm \sqrt{\Delta^{\mathbb{I}^2} - 4}}{2}$ of
 (\ref{DiracCharacteristic}) are called
{\it Floquet multipiers}. Similar reasoning as above may be applied to the equation
\begin{equation}
	Y(\pi) = \sigma_i\rho(\lambda) Y(0), \quad \mbox{for } i = 1,2,3,
\end{equation}  to obtain the $J,\sigma_1,\sigma_3$-discriminants which are
$\Delta^J$, $\nabla^I$ and $\nabla^J$, respectively. For brevity we refer to $\Delta^\mathbb{I}$ as $\Delta$. 

Let $\lambda \in \mathcal{S}= \{\lambda \in \R : |\Delta|\leq 2 \}$, there exist two linearly independent solutions of 
(\ref{eValProbFormal}) and
 (\ref{floquetEqn}) both of which have $|\rho|\leq 1$. The components of $\mathcal{S}$ are referred to as the {\it regions of stability}.
 Furthermore, the components of $\R\setminus\mathcal{S}$ are referred to as the {\it regions of instability}.  That these are suitable definitions will be apparent
from section \ref{sec-MainRes}.

The following lemmas are necessary for the inverse problem. The first such lemma follows the method in \cite[pg. 30]{FaU}, for Sturm-Liouville
 problems.

\begin{lem}\label{YeqYTLem}
	Let $\mathbb{Y}$ and $\tilde\mathbb{Y}$ be solutions of $\ell Y = \lambda Y$ satisfying the initial conditions
	(\ref{initialConds}) with canonical potentials $Q = Q_1$ and $\tilde{Q}=\tilde{Q_1}$, respectively.
	 If $\tilde\mathbb{Y}(\pi,\lambda) = \mathbb{Y}(\pi,\lambda)$, for all $\lambda\in\C$, then $\tilde\mathbb{Y}(z,\lambda) = \mathbb{Y}(z,\lambda)$,
	 for all $z\in \R$, $\lambda \in \C$.
\end{lem}
\begin{proof} Define the linear boundary operators, $U(Y) := y_2(0)$ and $V(Y) := y_2(\pi)$. Let $\Phi(z,\lambda)$ be defined by
	\begin{eqnarray}\label{Mfunct1.1}
		\Phi := Y_2 + MY_1,
	\end{eqnarray}
	 where $M$ is chosen so that $V(\Phi) = 0$. Note $U(\Phi) = 1$. Thus $M = -\frac{V(Y_2)}{V(Y_1)} = -\frac{y_{22}(\pi)}{y_{12}(\pi)}$.
	 Setting $Y_3 := y_{12}(\pi)Y_2 - y_{22}(\pi)Y_1$, we have
	\begin{equation}
		\begin{array}{cc}
		 Y_3 = \Delta_0 \Phi \quad \mbox{ where}  &
		 \Delta_0 :=  \mbox{Wron}[Y_1,{Y}_3] = y_{12}(\pi).
		\end{array}
	\end{equation}
Let $P(z,\lambda)$ be given by 
\begin{equation}\label{PDefn0.1}
	P(z,\lambda) \left( \begin{array}{cc}
 	 \tilde y_{11} & \tilde\Phi_1 \\
 \tilde y_{12}& \tilde\Phi_2 \end{array} \right)= \left( \begin{array}{cc}
y_{11} & \Phi_1 \\
 y_{12}& \Phi_2 \end{array} \right).		 
\end{equation}
Since $ \mbox{Wron}[\tilde{Y}_1, \tilde\Phi]  = 1$, a direct calculation gives
\begin{equation}\label{PDefn1.1}
	P(z,\lambda) = \left( \begin{array}{cc}
	  y_{11}\tilde\Phi_2 -\tilde y_{12}\Phi_1 &  \tilde y_{11} \Phi_1 - y_{11}\tilde\Phi_1  \\
	 y_{12}\tilde\Phi_2 - \tilde y_{12}\Phi_2 &  \tilde y_{11}\Phi_2 - y_{12}\tilde\Phi_1 \end{array} \right).
\end{equation}
Substituting (\ref{Mfunct1.1}) into the above equation gives
\begin{equation}\label{PDefn1.2}
	P(z,\lambda) = \left( \begin{array}{cc}
	  y_{11}\tilde y_{22} - y_{21}\tilde y_{12}&  y_{21} \tilde y_{11} - y_{11}\tilde y_{21}   \\
	 y_{12}\tilde y_{22} - y_{22}\tilde y_{12} &  y_{22}\tilde y_{11} -y_{12} \tilde y_{21}  \end{array} \right) +
	 (\tilde M - M)\left( \begin{array}{cc}
	  y_{11}\tilde y_{12}&  -y_{11}\tilde y_{11}  \\
	 y_{12}\tilde y_{12} &   -y_{12}\tilde y_{11}\end{array} \right).
\end{equation}
Since $\mathbb{Y}(\pi) = \tilde\mathbb{Y}(\pi)$, we have that $M(\lambda) = \tilde M(\lambda)$ for every $\lambda$, thus
 $P(z,\lambda)$ is entire for each $z\in\R$, as is $P(z,\lambda)-\mathbb{I}$. Combining (\ref{Mfunct1.1}) and (\ref{PDefn1.2}) with the identity
 $\Delta_0\mathbb{I} = \langle Y_1, Y_3\rangle\mathbb{I}$
gives  
\begin{equation}
	\Delta_0 (P(z,\lambda) - \mathbb{I}) = \left( \begin{array}{cc}
	    y_{11}(\tilde y_{32}-y_{32}) - y_{31}(\tilde y_{12}- y_{12})& y_{31}\tilde y_{11} - y_{11}\tilde y_{31}   \\
	  y_{12}\tilde y_{32}  -  y_{32}\tilde y_{12}& y_{32}(\tilde y_{11}-y_{11})  -  y_{12}(\tilde y_{31}- y_{31})   \end{array}
	 \right).
\end{equation}
Substituting the asymptotic expressions from Lemma \ref{LemAsymptotics} for $Y_1$ and $Y_3$ into the right hand side of the 
above equation gives
\begin{equation}\label{PDefnEqn1}
	\Delta_0 P(z,\lambda) = \Delta_0\mathbb{I} +  o(e^{|\Im \lambda|\pi}) \quad \mbox{ for } \lambda \in \C.
\end{equation}
We also note from Lemma \ref{LemAsymptotics}, the asymptotic estimate
\begin{equation}\label{Delta_0Asymp}
	\Delta_0 = -\sin\lambda\pi + o(e^{|\Im\lambda|\pi}).
\end{equation}
Define the sets $\tilde D_\epsilon^k := \{\lambda : |\sin\lambda\pi| <\epsilon ,|n - k|<\frac{1}{2} \}$.
Let $\tilde D_\epsilon = \cup_k \tilde{D}_\epsilon^k$ and notice that for large $|\lambda|$ and some $\epsilon >0$,
 $\tilde{D}_\epsilon^k$ contains exactly one zero of $\Delta_0$. Furthermore for large $|\lambda|$, 
\begin{equation}
	|\Delta_0|e^{-|\Im\lambda|\pi} \geq \epsilon  + o(1) \quad \mbox{ for every } \lambda \in \C\setminus\tilde{D}_\epsilon.
\end{equation}
This shows that for some $\epsilon >0$ there is a $C^* \in\R$ such that for large $|\lambda|$ we have 
\begin{equation}
	|\Delta_0| \geq C^*e^{|\Im\lambda|\pi} \quad \mbox{ for every }\lambda \in \C\setminus \tilde{D}_\epsilon.
\end{equation}
Thus
\begin{equation}\label{OneOverDeltaAsymp}
	\frac{1}{\Delta_0} = O(e^{-|\Im \lambda|\pi}) \quad \mbox{ for } \lambda \in \C \setminus \tilde{D}_\epsilon.
\end{equation}
Combining equations (\ref{PDefnEqn1}), (\ref{Delta_0Asymp}) and (\ref{OneOverDeltaAsymp}) gives that 
\begin{equation}\label{PAsymp2.1}
	P(z,\lambda) = \mathbb{I} + o(1) \quad \mbox{ for } \lambda \in \C \setminus \tilde{D}_\epsilon.
\end{equation}
The maximum modulus principle shows that relation (\ref{PAsymp2.1}) holds on $\C$.
 Thus $P$ is bounded on $\C$, hence by Liouville's Theorem, $P = \mathbb{I}$ on $\C$.
Finally, equation (\ref{PDefn0.1}) completes the Lemma. \qed
\end{proof}
\begin{lem}\label{CanonSystI}
	Suppose $Q$ in $\ell Y = \lambda Y$ is a $2 \times 2$ 
	$\pi$-periodic matrix function integrable on $[0, \pi)$, of the form
	$Q = Q_1$ then  $\Delta + 2$ has only double zeros  if and only if
	 $\mathbb{Y}(\pi) = \mathbb{Y}(\textstyle\frac{\pi}{2})^2$.
\end{lem}
\begin{proof}
	Assume $\mathbb{Y}(\pi) = \mathbb{Y}(\textstyle\frac{\pi}{2})^2$. A direct calculation gives 
	\begin{eqnarray}
		\mathbb{Y}^2(\textstyle\frac{\pi}{2}) &=& \frac{1}{4}(\Delta_+^\mathbb{I}\mathbb{I} + \Delta_+^JJ +
		 \nabla_+^\mathbb{I}\sigma_3 + \nabla_+^J\sigma_1)^2, \label{Ysquared2} \\
		&=& \frac{1}{4}((\Delta_+^\mathbb{I})^2 + (\nabla_+^J)^2 + (\nabla_+^\mathbb{I})^2 -
		 (\Delta_+^J)^2)\mathbb{I} +
		 \sum_{i=1}^3 d_i\sigma_i,
	\end{eqnarray}
	where $d_i$ are analytic functions of order $1$. Using equation (\ref{WronGL}) we have
	\begin{equation}\label{YPlusSqr12}
		\mathbb{Y}^2(\textstyle\frac{\pi}{2}) = \frac{1}{4}(2(\Delta_+^\mathbb{I})^2 - 4)\mathbb{I} + \sum_{i=1}^3
		 d_i\sigma_i.
	\end{equation}
	However, by assumption $\langle\mathbb{Y}(\textstyle\frac{\pi}{2})^2 -\mathbb{Y}(\pi), \mathbb{I} \rangle_{Lin} =
	 0$, thus using (\ref{YPlusSqr12})
	and the fact that the Pauli matrices form an orthonormal basis, gives
	\begin{equation}\label{traceG}
		(\Delta_+^\mathbb{I})^2 = \Delta + 2.
	\end{equation}
	 The above relation shows that the zeros of $\Delta + 2$ are at least of order $2$, but the maximal dimension of
	 every
	 eigenspace of $L_2$ is $2$. Thus $\Delta + 2$ has only double zeros.

	Conversely, assume $\Delta + 2$ has only double zeros. 
	$\mathbb{Y}(\pi)$ is an entire matrix valued function of order $1$, thus $\Delta + 2$
	is an entire function of order $1$.

	At every double zero, $\lambda = \tilde\lambda$, of $\Delta + 2$,
	 the corresponding instability interval vanishes, furthermore the eigenspace of $L_2$ is of dimension $2$, thus every solution is $\pi$ anti-periodic, giving
	\begin{eqnarray}
		F(z,\lambda) := \mathbb{Y}(z + \pi) + \mathbb{Y}(z) = 0. 
	\end{eqnarray}
	This condition is also necessary for an anti-periodic eigenvalue to be double.
	Since $\Delta + 2$ is an entire function of order $1$ with all zeros being double, it follows from the Hadamard
	expansion of $\Delta + 2$ as an infinite product that $\sqrt{\Delta + 2}$ is an entire function of order $\textstyle\frac{1}{2}$ with all
	zeros simple. Now $F(z,\lambda)$ is an entire function of order $1$, and the zeros of $\sqrt{\Delta + 2}$ and   	
 	$F(z,\lambda)$ coincide. Thus $\frac{F(z,\lambda)}{\sqrt{\Delta + 2}}$ is an entire function.

	Lemma \ref{LemAsymptotics} and (\ref{DiscrimDef01}) give
	\begin{equation}\label{DeltaP2Asymptotic}
		\Delta + 2 = 2\cos\lambda\pi + 2 +  o\left(e^{|\Im \lambda|\pi}\right), \label{DeltaAsymp}
	\end{equation}
	thus 
	\begin{equation}
		|\Delta + 2|e^{-|\Im \lambda|\pi} = \left|(2\cos\lambda\pi + 2)e^{-|\Im \lambda|\pi} +
		 o(1)\right|.
	\end{equation}
	Define the sets
	\begin{equation}
		D^k_\epsilon := \{ \lambda : |2\cos\lambda\pi + 2| < \epsilon,
		 |\lambda - (2k+1)| < \textstyle\frac{1}{2} \},
	\end{equation}
	for each $k \in \Z$ and a fixed $\epsilon > 0$ so small so that every $D^k_\epsilon$ is a single simply connected set.
	For brevity we write $D_{\epsilon} = \cup_k D^k_\epsilon$ and note that for large $|\lambda|$ each $D^k_\epsilon$
	contains a exactly one zero of $\Delta + 2$.
	For $\lambda \in \C \setminus D_\epsilon$, large $|\Im\lambda|$, we have $|2\cos\lambda\pi +
	 2|e^{-|\Im
	 \lambda|\pi} \geq \frac{1}{2}$. For $C>0$ and $\lambda \in \C \setminus D_\epsilon$ with $|\Im\lambda| \leq C$ and for large $|\Re\lambda|$ we have
	 $|2\cos\lambda\pi + 2|e^{-|\Im
	 \lambda|\pi} \geq \epsilon e^{-c\pi}$.  Thus there exists a $k > 0$ so large 
	 that 
	\begin{equation}
		|\Delta + 2|e^{-|\Im \lambda|\pi} \geq  \min\{ \textstyle\frac{1}{2}, \epsilon
		 e^{-c\pi} \}  + o(1), \quad \mbox{ for all } \quad |\lambda| \geq k,
	\end{equation}
	for $\lambda \in \C \setminus D_\epsilon$. Hence
	\begin{equation}\label{DeltaAsymp2}
		\frac{1}{\sqrt{\Delta + 2}} = O\left(e^{-|\Im\lambda|\frac{\pi}{2}}\right) \quad \mbox{ for} \quad
		\lambda \in \C \setminus D_\epsilon. \label{asymp1.2}
	\end{equation}
	Lemma \ref{LemAsymptotics} and $e^{-\lambda J \frac{\pi}{2}} +  e^{\lambda J \frac{\pi}{2}} = 2\mathbb{I}\cos(\lambda
	 \frac{\pi}{2})$ yield
	\begin{eqnarray}
		F(-\frac{\pi}{2},\lambda) &=& e^{-\lambda J \frac{\pi}{2}} +  e^{\lambda J \frac{\pi}{2}} +
		 o(e^{|\Im\lambda|\frac{\pi}{2}}), \\
									&=& 2\mathbb{I}\cos(\lambda\textstyle\frac{\pi}{2})  +
									 o\left(e^{|\Im\lambda|\frac{\pi}{2}}\right). 
		 \label{Fasymp1.1}
	\end{eqnarray}
	Combining (\ref{DeltaAsymp2}) and (\ref{Fasymp1.1}) yields 
	\begin{eqnarray}\label{FAsymp4}
		\tilde F := \frac{F(-\frac{\pi}{2},\lambda)}{\sqrt{\Delta + 2}} = O(1), \label{tildeFasymp}
	\end{eqnarray}
	for $\lambda \in \C \setminus D_\epsilon$. However $\tilde F$ is entire in $\C$, so the maximum modulus principle
	 gives that $\tilde{F} = O(1)$ in $\C$, thus is constant in $\C$ by Liouville's Theorem. So there exists $a,b,c,d \in \mathbb{C}$
	 such that 
	\begin{equation}\label{FFunctConstLiou}
		\frac{F(-\frac{\pi}{2},\lambda)}{\sqrt{\Delta + 2}} = \left( \begin{array}{cc} a  & b  \\
			  c  &  d \end{array} \right), \quad \mbox{ for all } \lambda \in \mathbb{C}.
	\end{equation}
	
	For $\lambda = i\zeta$, $\zeta \rightarrow \infty$, equation (\ref{DeltaAsymp}) gives
		$\frac{1}{\Delta + 2} = e^{-\pi\zeta}(1 + o(1))$, thus
		$\frac{1}{\sqrt{\Delta + 2}} = {e^{-\frac{\pi}{2}\zeta}}(1 + o(1)))$. Furthermore,
		 (\ref{Fasymp1.1}) gives
		\begin{eqnarray}
			F(-\frac{\pi}{2},i \zeta) &=& e^{\frac{\pi\zeta}{2}}(\mathbb{I} + o(1)), \label{finalAsym1}
		\end{eqnarray}
	hence $a =1 = d$ and $c = 0 = b$, thus 
	\begin{equation}\label{SumYDelta1}
		\mathbb{Y}(\textstyle\frac{\pi}{2}) + \mathbb{Y}(-\textstyle\frac{\pi}{2}) = \mathbb{I}\sqrt{\Delta + 2}. \label{Sum1}
	\end{equation}

So the analogues of (\ref{GLYLin}) at $\mathbb{Y}(\textstyle\frac{\pi}{2})$ and $\mathbb{Y}(-\textstyle\frac{\pi}{2})$
combined with (\ref{SumYDelta1}) give
\begin{equation}
	(\Delta^\mathbb{I}_{+} + \Delta^\mathbb{I}_{-})\mathbb{I} + (\Delta^J_{+} + \Delta^J_{-})J +
	 (\nabla^\mathbb{I}_{+} +
	 \nabla^\mathbb{I}_{-})\sigma_3 + (\nabla^J_{+} + \nabla^J_{-})\sigma_1
	= 2\mathbb{I}\sqrt{\Delta + 2}.
\end{equation}
Applying the inner product $\langle \cdot, \sigma_i \rangle_{Lin}$, $i = 0,...,3$, to both sides of the above equation gives
\begin{eqnarray}
	\Delta^\mathbb{I}_{+} + \Delta^\mathbb{I}_{-} &= 2\sqrt{\Delta + 2}, \label{DI} \\ 
	\Delta^J_{+} &= -\Delta^J_{-}, \label{DJ}\\
	\nabla^\mathbb{I}_{+} &= -\nabla^\mathbb{I}_{-},\label{NI} \\
	\nabla^J_{+} &= -\nabla^J_{-}.\label{NJ}
\end{eqnarray}
Furthermore, equations (\ref{WronGL}), (\ref{DJ}), (\ref{NI}) and (\ref{NJ}) gives $\Delta^\mathbb{I}_{+} = \pm \Delta^\mathbb{I}_{-}$,
but if $\Delta^\mathbb{I}_{+} = - \Delta^\mathbb{I}_{-}$, equation (\ref{DI}) shows that $\sigma(L_2) = \C$, which is
 not
possible as $L_2$ is a self-adjoint operator and thus has $\sigma(L_2) \subset \R$. Hence
\begin{equation}
	\Delta^\mathbb{I}_{+} =  \Delta^\mathbb{I}_{-}. \label{DI0}
\end{equation} 

A direct calculation shows that 
\begin{equation}\label{FMInverse1}
	\mathbb{Y}^{-1}(-\textstyle\frac{\pi}{2}) = \frac{1}{2}(\Delta^\mathbb{I}_{+}\mathbb{I} -\Delta^J_{-}J - \nabla^\mathbb{I}_{-}\sigma_3 -
	 \nabla^J_{-}\sigma_1).
\end{equation}
Applying (\ref{DJ})-(\ref{DI0}) to (\ref{FMInverse1}) shows that 
$\mathbb{Y}^{-1}(-\frac{\pi}{2}) = \mathbb{Y}(\frac{\pi}{2})$. Since $Q$ is $\pi$-periodic, the fundamental
 matrix solutions
 $\mathbb{Y}(z)$ and 
$\mathbb{Y}(z + \pi)$ of (\ref{eValProbFormal}) are related by
\begin{equation}\label{SolnTrans09}
	\mathbb{Y}(z + \pi) = \mathbb{Y}(z)\mathbb{Y}(\pi).
\end{equation}  
Setting $z = -\frac{\pi}{2}$ in (\ref{SolnTrans09}) gives
\begin{equation}\label{FMSTrans4}
	\mathbb{Y}(\textstyle\frac{\pi}{2}) = \mathbb{Y}(-\textstyle\frac{\pi}{2})\mathbb{Y}({\pi}),
\end{equation}
thus
\begin{equation}
	\mathbb{Y}(\pi) = \mathbb{Y}(\textstyle\frac{\pi}{2})^2. \label{suff02}
\end{equation}
	\qed
\end{proof}
\begin{lem}\label{CanonSystI2}
	Suppose $Q$ in $\ell Y = \lambda Y$ is a $2 \times 2$ 
	$\pi$-periodic matrix function integrable on $[0, \pi)$, of the form
	$Q = Q_1$
	then  $\Delta - 2$ has only double zeros if and only if
	\mbox{$\mathbb{Y}(\pi) = (\sigma_2 \mathbb{Y}(\textstyle\frac{\pi}{2}))^2$}.
\end{lem}
\begin{proof}
	Let us assume that $\mathbb{Y}(\pi) =
	 (\sigma_2\mathbb{Y}(\textstyle\frac{\pi}{2}))^2$. We have 
	\begin{equation}\label{SigmaY2eqn}
		\sigma_2 \mathbb{Y}(\textstyle\frac{\pi}{2}) = \frac{1}{2}(\Delta^\mathbb{I}_{+}\sigma_2 -i\Delta^J_{+}\mathbb{I} -
		 i\nabla^\mathbb{I}_{+}\sigma_1 +
		 i\nabla^J_{+}\sigma_3),
	\end{equation}
	thus using (\ref{WronGL}) a direct computation gives
	\begin{equation}
		(\sigma_2 \mathbb{Y}(\textstyle\frac{\pi}{2}))^2 = \frac{1}{4}(4 - 2(\Delta^J_+)^2)\mathbb{I} + \sum_{i=1}^3 d_i\sigma_i,
	\end{equation}
	where $d_i$ are analytic functions of order $1$. Considering that the Pauli matrices form an orthonormal set, using
	$\mathbb{Y}(\pi) -
	 (\sigma_2\mathbb{Y}(\textstyle\frac{\pi}{2}))^2 = 0$, we calculate 
	$\langle \mathbb{Y}(\pi) - (\sigma_2\mathbb{Y}(\textstyle\frac{\pi}{2}))^2, \mathbb{I} \rangle_{Lin} = 0$, to find
	\begin{equation}\label{Deltamin2doubleZ}
		(\Delta_+^J)^2 = 2 -\Delta^\mathbb{I}.
	\end{equation}
	The zeros of $(\Delta_+^J)^2$ are at least of order $2$, however the maximal dimension of every
	 eigenspace of $L_1$ is $2$. Thus all the zeros of $\Delta - 2$ are double.
	
	For sufficiency, assume that all the zeros of $\Delta - 2$ are double. Define
	\begin{equation}
		H(x,\lambda) := \mathbb{Y}(x + \pi) - \mathbb{Y}(x).
	\end{equation}
	Using similar reasoning to Lemma \ref{CanonSystI}, we have that $\frac{H(z,\lambda)}{\sqrt{2 - \Delta}}$ is an entire
	 function, thus we have 
	\begin{equation}\label{HfunctionAsymp2}
		H(-\textstyle\frac{\pi}{2},\lambda) = 2J\sin(\lambda\textstyle\frac{\pi}{2}) + o(e^{|\Im\lambda|\frac{\pi}{2}}).
	\end{equation}
	
	Lemma \ref{LemAsymptotics} and (\ref{DiscrimDef01}) give
	\begin{equation}\label{DeltaP2Asymptotic3}
		2 - \Delta = 2 - 2\cos\lambda\pi +  o\left(e^{|\Im \lambda|\pi}\right), 
	\end{equation}
	thus 
	\begin{equation}
		|2 - \Delta|e^{-|\Im \lambda|\pi} = \left|(2 - 2\cos\lambda\pi)e^{-|\Im \lambda|\pi} +
		 o(1)\right|.
	\end{equation}
	Define the sets
	\begin{equation}
		\hat{D}^k_\epsilon := \{ \lambda : |2\cos\lambda\pi - 2| < \epsilon,
		 |\lambda - 2k| < \textstyle\frac{1}{2} \},
	\end{equation}
	for each $k \in \Z$ and a fixed $\epsilon > 0$ so small so that every $\hat{D}^k_\epsilon$ is a single simply connected
	 set.
	For brevity we write $\hat{D}_{\epsilon} = \cup_k \hat{D}^k_\epsilon$ and note that for large $|\lambda|$ each
	 $\hat{D}^k_\epsilon$ contains a exactly one zero of $\Delta - 2$.
	Following reasoning as in Lemma \ref{CanonSystI} we have 
	\begin{equation}\label{DeltaAsymp3}
		\frac{1}{\sqrt{2 - \Delta}} = O\left(e^{-|\Im\lambda|\frac{\pi}{2}}\right) \quad \mbox{ for } \quad
		\lambda \in \C \setminus \hat{D}_\epsilon. 
	\end{equation} 
	Thus
	\begin{equation}
		\tilde{H} := \frac{H(-\textstyle\frac{\pi}{2},\lambda)}{\sqrt{2 - \Delta}} = O(1),
	\end{equation}
	for $\lambda \in \C \setminus \hat{D}_\epsilon$.
	Since $\tilde{H}$ is entire on $\C$, the maximum-modulus theorem shows that it is bounded on $\C$. Hence by Liouville's
	 theorem $\tilde{H}$ is constant on $\C$.
	For $\lambda = i\zeta$, $\zeta \rightarrow \infty$, equation (\ref{DeltaP2Asymptotic3}) gives 
	$\frac{1}{\sqrt{2 - \Delta}} = e^{-\zeta\frac{\pi}{2}}(-i  + o(1))$, also equation 
	(\ref{HfunctionAsymp2}) gives
	\begin{equation}
		H(-\textstyle\frac{\pi}{2},\lambda) =  e^{\zeta \frac{\pi}{2}}(iJ + o(1)),
	\end{equation}
	thus 
	\begin{equation}
		\frac{H(-\textstyle\frac{\pi}{2},\lambda)}{\sqrt{2 - \Delta}} =  J + o(1).
	\end{equation}
	Giving
	\begin{equation}\label{SCharEqn443}
		\mathbb{Y}(\textstyle\frac{\pi}{2}) - \mathbb{Y}(-\textstyle\frac{\pi}{2}) =  J\sqrt{2 - \Delta}.
	\end{equation}
	Similarly to Lemma \ref{CanonSystI}, the analogues of (\ref{GLYLin}) at $\mathbb{Y}(\textstyle\frac{\pi}{2})$ and $\mathbb{Y}(-\textstyle\frac{\pi}{2})$
	combined with (\ref{SCharEqn443}), give the expansion
	\begin{equation}
		(\Delta^\mathbb{I}_{-} - \Delta^\mathbb{I}_{+})\mathbb{I} + (\Delta^J_{-} - \Delta^J_{+})J + (\nabla^\mathbb{I}_{-} -
		 \nabla^\mathbb{I}_{+})\sigma_3 + (\nabla^J_{-} + \nabla^J_{+})\sigma_1
		= 2 J\sqrt{2 -\Delta}.
	\end{equation}
	Applying the inner product $\langle \cdot, \sigma_i \rangle_{Lin}$ to the above equation yields
	\begin{eqnarray}
		\Delta^J_{-} - \Delta^J_{+} &= 2\sqrt{2 -\Delta}, \label{DI2} \\ 
		\Delta^\mathbb{I}_{-} &= \Delta^\mathbb{I}_{+}, \label{DJ2}\\
		\nabla^\mathbb{I}_{-} &= \nabla^\mathbb{I}_{+},\label{NI2} \\
		\nabla^J_{-} &= \nabla^J_{+}.\label{NJ2}
	\end{eqnarray}
	The identity (\ref{WronGL}) together with (\ref{DJ2})-(\ref{NJ2}) gives 
	$\Delta^J_{-} = -\Delta^J_{+}$, otherwise (\ref{DI2}) yields $\sigma(L_1) = \C$,  which is not possible since $L_1$ is
	self-adjoint. A direct calculation shows that 
	\begin{equation}
		\sigma_2 \mathbb{Y}(\textstyle\frac{\pi}{2})\sigma_2 = \frac{1}{2}(\Delta^\mathbb{I}_{+}\mathbb{I} +
		 \Delta^J_{+}J -
		 \nabla^\mathbb{I}_{+}\sigma_3 - \nabla^J_{+}\sigma_1).
	\end{equation}
	Comparing the above equation with (\ref{FMInverse1}) and (\ref{DJ2})-(\ref{NJ2}) shows that
	$\sigma_2 \mathbb{Y}(-\frac{\pi}{2})\sigma_2 = \mathbb{Y}(\frac{\pi}{2})^{-1}$. Thus using (\ref{FMSTrans4})
	 we have
	\begin{equation}
		\mathbb{Y}(\pi) = (\sigma_2\mathbb{Y}(\textstyle\frac{\pi}{2}))^2.
	\end{equation}
	 \qed

\end{proof}
%%%%%%%%%%%%%%%%%%%%%%%%%% Main Results %%%%%%%%%%%%%%%%%%%%%%%%%%%%%%%%%%%%%%%%%%%%%%%%%%%%%%%
\newsection{Main results \label{sec-MainRes}}

We are now in a position to prove our main theorems. Let
\begin{eqnarray}
	R(z) &:=& e^{J\int_0^z (Q_2 - \frac{1}{\pi}\int_0^\pi Q_2 d\tau) dt}, \label{RfunctionDefn}
\end{eqnarray}
thus $Y = R\tilde{Y}$ transforms $\ell Y = \lambda Y$ into
\begin{equation}\label{QEqn1}
	J\tilde{Y}' + \tilde{Q}\tilde{Y} = \lambda\tilde{Y},
\end{equation} 
where $\tilde{Q} = \tilde{Q_1} + \tilde{Q_2}$ in which
\begin{eqnarray}
	\tilde{Q}_1(z) &=& R^{-1}(z)Q_1(z) R(z),  \label{Q1Relation21} \\
	\tilde{Q}_2(z) &=& \frac{1}{\pi}\int_0^\pi Q_2 dt, \label{Q1Relation22}
\end{eqnarray}
with $\tilde{Q}_1 \in S_{-}^J$ and $\tilde{Q}_2 \in S_{+}^J$.
Notice that $R(0) = \mathbb{I}= R(\pi)$, thus the above transformation preserves boundary conditions. If we consider the equation
\begin{equation}
	J\tilde{Y_a}' + \tilde{Q_1}\tilde{Y_a} =  \left(\lambda - \frac{1}{2\pi}\int_0^\pi (q_1 + q_2) dt\right)\tilde{Y_a},
\end{equation}
then
\begin{equation}\label{Ytransf43}
	\mathbb{Y}(\lambda,z) = R(z)e^{-i\int_0^z\Im q dt}\tilde\mathbb{Y}_a\left(\lambda - \frac{1}{2\pi}\int_0^z (q_1 + q_2) dt,z\right).
\end{equation}
 Setting $x = \pi$ in equation (\ref{Ytransf43}) and taking the trace we have
\begin{equation}\label{discrimrelation1}
	\Delta(\lambda) = e^{-i\int_0^\pi \Im q dt}\tilde\Delta_a\left(\lambda - \frac{1}{2\pi}\int_0^\pi (q_1 + q_2) dt\right).
\end{equation}
Equation (\ref{discrimrelation1}) shows that $\Delta$ maps $\lambda \in \R$ into the line $\{ \eta e^{-i\int_0^\pi \Im q dt} : \eta \in \R \}$.

We say that $Q$ is {\bf $\frac{\pi}{2}$-$\sigma_2$-similar} if $Q(x + \frac{\pi}{2}) = \sigma_2 Q(x) \sigma_2$, this is equivalent
 to $Q_1$ being $\frac{\pi}{2}$-anti-periodic and $Q_2$ being $\frac{\pi}{2}$-periodic.

\begin{thm}\label{perthm}
	Suppose $Q$ in $\ell Y = \lambda Y$ is a Hermitian $2 \times 2$ complex 
	$\pi$-periodic matrix function integrable on $[0, \pi)$, 
	 then the following hold: \\
	{\bf (a)} If $Q$ is a.e. $\frac{\pi}{2}$-periodic then $\Delta + 2e^{-i\int_0^\pi \Im q dt}$ has only double zeros. \\
	{\bf (b)} If $\Delta + 2e^{-i\int_0^\pi \Im q dt}$ has only double zeros then $\tilde{Q}$ is a.e. $\frac{\pi}{2}$-periodic,
	where $\tilde{Q}$ is as given in (\ref{QEqn1})-(\ref{Q1Relation22}).
	
\end{thm}

\begin{thm}\label{antiperthm}
	Suppose  $Q$ in $\ell Y = \lambda Y$ is a Hermitian $2 \times 2$ complex $\pi$-periodic 
	matrix function integrable on $[0, \pi)$,  then the following hold: \\	
	{\bf (a)} If $Q_1$ is a.e. $\frac{\pi}{2}$-anti-periodic and $Q_2$ is a.e. $\frac{\pi}{2}$-periodic
	 then $\Delta - 2e^{-i\int_0^\pi \Im q dt}$ has only double zeros. \\
	{\bf (b)} If $\Delta - 2e^{-i\int_0^\pi \Im q dt}$ has only double zeros then $\tilde{Q}_1$ is a.e.
	 $\frac{\pi}{2}$-anti-periodic and $\tilde{Q}_2$ is a.e. $\frac{\pi}{2}$-periodic, where $\tilde{Q}_1$ and $\tilde{Q}_2$
	 are as given in (\ref{QEqn1})-(\ref{Q1Relation22})
\end{thm}

If $Q$ is in a canonical form then $R = \mathbb{I}$ so that $Q = Q_1 = \tilde{Q}_1$. This leads to the following 
Corollary to Theorems \ref{perthm} and \ref{antiperthm}.

\begin{cor} \label{BorgUnique}
	If $Q$ in $\ell Y = \lambda Y$ is a $2 \times 2$ canonical $\pi$-periodic 
	matrix function integrable on $[0, \pi)$ then the following hold:\\
	{\it (a)} $Q$ is a.e. $\frac{\pi}{2}$-periodic if and only if $\Delta + 2$ has only double zeros.\\
	{\it (b)} $Q$ is a.e. $\frac{\pi}{2}$-anti-periodic if and only if $\Delta - 2$ has only double zeros.\\
\end{cor}
\begin{cor} \label{Ambarzumyan} (Ambarzumyan)
	If $Q$ in $\ell Y = \lambda Y$ is a Hermitian $2 \times 2$ complex $\pi$-periodic 
	matrix function integrable on $[0, \pi)$, then 
	every instability interval vanishes if and only if $Q = r\sigma_0 + q\sigma_2$  a.e., where $r$ and $q$ are
	 real and integrable on $[0,\pi)$.
\end{cor}

The following example shows that the converse of {\bf (a)} in Theorem \ref{perthm} is not possible in general.
\begin{example}
	Suppose $Q = Q_1$ is a.e. $\frac{\pi}{2}$-periodic and consider the transformation $Y = \hat{R}\hat{Y}$, where
	\begin{equation}
		\hat{R}(z) = e^{J\int_0^z (-2t + \pi)(\mathbb{I} + \sigma_2)dt}.
	\end{equation} 
\end{example}
From Theorem \ref{perthm} {\bf (a)} we have that the zeros of $\Delta + 2$ are all double. Furthermore $-2x + \pi$ has mean value
 zero on $[0,\pi]$, thus $\hat{R}(0) = \hat{R}(\pi) = \mathbb{I}$, so that $\hat{R}$ preserves the boundary conditions. The transformation
 $Y = \hat{R}\hat{Y}$ gives
\begin{equation}
	J\hat{Y}' + (\hat{Q}_1 + \hat{Q}_2)\hat{Y}  = \lambda \hat{Y},
\end{equation}
where 
\begin{eqnarray}
	\hat{Q}_1(z) = e^{2J(z^2 - \pi z )}Q_1(z), \\
	\hat{Q}_2(z) = (-2z + \pi)(\mathbb{I} + \sigma_2).
\end{eqnarray}
Notice that $\hat{Q}_2(\frac{\pi}{4}) = \frac{\pi}{2}(\mathbb{I} + \sigma_2)$ while 
$\hat{Q}_2(\frac{3\pi}{4}) = -\frac{\pi}{2}(\mathbb{I} + \sigma_2)$, thus $\hat{Q}_2$ is not $\frac{\pi}{2}$-periodic even though
zeros of $\hat\Delta + 2$ are all double.

{\bf Proof of Theorem 5.1:} To prove {\bf (a)}, assume $Q$ is $\frac{\pi}{2}$-periodic. The fundamental solutions 
 $\mathbb{Y}(z+ \pi)$ and $\mathbb{Y}(z + \frac{\pi}{2})$ are both solutions of
 $\ell Y = \lambda Y$ thus
\begin{equation}\label{Qp/2per1}
	\mathbb{Y}(z+ \pi) = \mathbb{Y}(z + \textstyle\frac{\pi}{2})B, \label{SolnTrans01}
\end{equation}
for some invertible matrix $B$, which may depend on $\lambda$.
Setting $z = -\frac{\pi}{2}$ in the above equation gives $B = \mathbb{Y}(\frac{\pi}{2})$, thus
\begin{eqnarray} \label{GP11}
	\mathbb{Y}(z + \textstyle\frac{\pi}{2}) = \mathbb{Y}(z)\mathbb{Y}(\textstyle\frac{\pi}{2})
	 \quad
	 \mbox{ and }
	 \quad
	\mathbb{Y}(z+ \pi) = \mathbb{Y}(z + \frac{\pi}{2})\mathbb{Y}(\frac{\pi}{2}).
\end{eqnarray}
Setting $z = 0$ in the second equation of (\ref{GP11}) gives
\begin{equation} \label{PeriodicFMS1}
	\mathbb{Y}(\pi) = \mathbb{Y}(\textstyle\frac{\pi}{2})^2.
\end{equation}
Since $Q_2$ is $\frac{\pi}{2}$-periodic, a direct calculation shows that $R(\pi) = \mathbb{I} = R(\frac{\pi}{2})$, for $R(z)$ defined by (\ref{RfunctionDefn}). Thus (\ref{Ytransf43}) and (\ref{PeriodicFMS1}) give
\begin{equation}\label{periodicFMS43}
	\tilde\mathbb{Y}_a(\pi) = \tilde\mathbb{Y}_a(\textstyle\frac{\pi}{2})^2.
\end{equation}
Furthermore, 
following the method used in (\ref{Ysquared2})-(\ref{traceG}) we obtain 
\begin{equation}\label{discrimeqn453}
(\tilde\Delta_{a+}^\mathbb{I})^2 = \tilde\Delta_a + 2.
\end{equation}
Equation (\ref{discrimeqn453}) shows that $\tilde\Delta_a + 2$ has only zeros of order $2n$, $n\in\N$, 
but the maximal dimension of the eigenspace of $\sigma(L_2)$ is $2$, thus $\tilde\Delta_a + 2$ has only double zeros.
Hence equation (\ref{discrimrelation1}) shows that $e^{i\int_0^\pi \Im q dt}\Delta + 2$ has only double zeros.

For {\bf (b)}, suppose $e^{i\int_0^\pi \Im q dt}\Delta + 2$ has only double zeros, thus $\tilde\Delta_a + 2$
has only double zeros. From Lemma \ref{CanonSystI} we have 
\begin{equation}\label{PeriodicFMS1a}
	\tilde\mathbb{Y}_a(\pi) = \tilde\mathbb{Y}_a(\textstyle\frac{\pi}{2})^2.
\end{equation}
Consider the problem
\begin{equation}
	J\tilde{Y_b}' + \tilde{Q_b}\tilde{Y_b} =  \left(\lambda - \frac{1}{2\pi}\int_0^\pi (q_1 + q_2) dt\right)\tilde{Y_b},
\end{equation}
where $\tilde{Q}_b( x ) := \tilde{Q}_1(x \mbox{ mod } \frac{\pi}{2})$ a.e., where $x \mbox{ mod } \frac{\pi}{2} \in [0, \frac{\pi}{2})$ for all $x \in \mathbb{R}$.
 It follows that $\tilde{Q}_b$ is a.e.
 $\frac{\pi}{2}$-periodic,
 then proceeding as in (\ref{Qp/2per1})-(\ref{periodicFMS43}) we have 
\begin{equation}\label{PeriodicFMS1b}
	\tilde\mathbb{Y}_b(\pi) = \tilde\mathbb{Y}_b(\textstyle\frac{\pi}{2})^2.
\end{equation}
However, by construction $\tilde\mathbb{Y}_b(\frac{\pi}{2}) = \tilde\mathbb{Y}_a(\frac{\pi}{2})$, thus 
(\ref{PeriodicFMS1a}) and (\ref{PeriodicFMS1b}) show that $\tilde\mathbb{Y}_b(\pi) = \tilde\mathbb{Y}_a(\pi)$. Using Lemma
 \ref{YeqYTLem} we have that
 $\tilde\mathbb{Y}_b(\lambda, x) = \tilde\mathbb{Y}_a(\lambda, x)$, for $\lambda \in \C$, $x \in \R$. Thus 
as
\begin{eqnarray}\label{yby1sea}
	\tilde{Q}_b - \tilde{Q}_1 = J(\tilde{\mathbb{Y}}_a' \tilde{\mathbb{Y}}_a^{-1} - \tilde{\mathbb{Y}}_b' \tilde{\mathbb{Y}}_b^{-1}) = 0,
\end{eqnarray}
we have $\tilde{Q}_b = \tilde{Q}_1$, and $\tilde{Q}_1$ is a.e.
$\frac{\pi}{2}$-periodic. Since $\tilde{Q}_2$ is constant, we have that $\tilde{Q}$ is a.e. $\frac{\pi}{2}$-periodic.
 \qed
%%%%%%%%%%%%%%%%%%% theorem 5.2 %%%%%%%%%%%%%%%%%%%%%%%%%%%%%%%%%

{\bf Proof of Theorem 5.2:} To prove {\bf (a)}, assume that $Q_1$ is a.e. $\frac{\pi}{2}$-anti-periodic and $Q_2$ is a.e.
 $\frac{\pi}{2}$-periodic, then $\mathbb{Y}(x)$ and $\sigma_2
 \mathbb{Y}(x + \frac{\pi}{2})$ are both solutions of $\ell Y = \lambda Y$, thus they are related by a
 transformation matrix $B$ as
\begin{equation}\label{thm5.2a}
	\sigma_2\mathbb{Y}(z + \textstyle\frac{\pi}{2}) = \mathbb{Y}(z)B.
\end{equation}
Setting $z = 0$ in the above equation gives $B = \sigma_2 \mathbb{Y}(\frac{\pi}{2})$, thus
\begin{equation}
	\mathbb{Y}(z + \textstyle\frac{\pi}{2}) = \sigma_2\mathbb{Y}(z)\sigma_2\mathbb{Y}(\textstyle\frac{\pi}{2}).
\end{equation}
At $z = \frac{\pi}{2}$ we have 
\begin{equation}\label{Sigma2DefEqn}
	\mathbb{Y}(\pi) = (\sigma_2\mathbb{Y}(\textstyle\frac{\pi}{2}))^2. \label{FSS2}
\end{equation}
Since $Q_2$ is $\frac{\pi}{2}$-periodic we have that $R(\pi) = \mathbb{I} = R(\frac{\pi}{2})$ for $R(z)$ defined by (\ref{RfunctionDefn}). Thus (\ref{Ytransf43}) and (\ref{PeriodicFMS1}) give
\begin{equation}\label{periodicFMS432}
	\tilde\mathbb{Y}_a(\pi) = \tilde\mathbb{Y}_a(\textstyle\frac{\pi}{2})^2.
\end{equation}
Following the method used in (\ref{SigmaY2eqn})-(\ref{Deltamin2doubleZ}), we have 
\begin{equation}
	(\tilde\Delta_{+a}^J)^2 = 2 -\tilde\Delta_a^\mathbb{I}.
\end{equation} 
 The above equation shows that $\tilde\Delta_a - 2$ has only zeros of order $2n$, $n\in\N$, but the maximal dimension of the 
eigenspace of $\sigma(L_1)$ is $2$, thus $\tilde\Delta_a - 2$ has only double zeros. Combining this with (\ref{discrimrelation1})
proves {\bf (a)}.

For {\bf (b)}, suppose $e^{i\int_0^\pi \Im q dt}\Delta - 2$ has only double zeros, thus $\tilde\Delta_a - 2$
has only double zeros. From Lemma \ref{CanonSystI2} we have 
\begin{equation}\label{PeriodicFMS2a}
	\tilde\mathbb{Y}_a(\pi) = (\sigma_2\tilde\mathbb{Y}_a(\textstyle\frac{\pi}{2}))^2.
\end{equation}
Consider the problem
\begin{equation}
	J\tilde{Y_b}' + \tilde{Q_b}\tilde{Y_b} =  \left(\lambda - \frac{1}{2\pi}\int_0^\pi (q_1 + q_2) dt\right)\tilde{Y_b},
\end{equation}
where $\tilde{Q}_b( x ) := \tilde{Q}_1(x \mbox{ mod } \frac{\pi}{2})$ a.e., where $x \mbox{ mod } \frac{\pi}{2} \in [0, \frac{\pi}{2})$ for all $x \in \mathbb{R}$ to be $\frac{\pi}{2}$-anti-periodic,
 then following (\ref{thm5.2a})-(\ref{periodicFMS432}) we have 
\begin{equation}\label{PeriodicFMS2b}
	\tilde\mathbb{Y}_b(\pi) = (\sigma_2\tilde\mathbb{Y}_b(\textstyle\frac{\pi}{2}))^2.
\end{equation}
However, by construction $\tilde\mathbb{Y}_b(\frac{\pi}{2}) = \tilde\mathbb{Y}_a(\frac{\pi}{2})$, thus 
(\ref{PeriodicFMS2a}) and (\ref{PeriodicFMS2b}) show that $\tilde\mathbb{Y}_b(\pi) = \tilde\mathbb{Y}_a(\pi)$. Using Lemma
 \ref{YeqYTLem} we have that
 $\tilde\mathbb{Y}_b(\lambda, x) = \tilde\mathbb{Y}_a(\lambda, x)$, for $\lambda \in \C$, $x \in \R$. Thus 
as in the proof of Theorem \ref{perthm} equation (\ref{yby1sea}) gives $\tilde{Q}_b = \tilde{Q}_1$, and $\tilde{Q}_1$ is a.e.
$\frac{\pi}{2}$-anti-periodic. Since $\tilde{Q}_2$ is constant, we have that $\tilde{Q}_2$ is a.e. $\frac{\pi}{2}$-periodic.
 \qed

{\bf Proof of Corollary \ref{Ambarzumyan}:} 
Assuming that $Q = r\sigma_0 + q\sigma_2$ a.e., we may rewrite equation (\ref{DifferentialExpr}) as $Y' = (pJ 
 -iq\mathbb{I}   - \lambda J)Y$, thus
$\mathbb{Y}(x) = e^{ J\int_0^z p dt -i \mathbb{I}\int_0^z q dt - J\lambda z}$, so that
\begin{equation}
	\Delta = 2\cos\left(\lambda \pi - \int_0^\pi p dt\right)e^{-i \mathbb{I}\int_0^\pi q dt}.
\end{equation}
The above equation shows that $|\Delta|\leq 2$, thus every instability interval vanishes. 

For necessity, assume that every instability interval vanishes, thus for any fixed $e^{i\int_0^\pi \Im q dt}$ 
all zeros of $e^{i\int_0^\pi \Im q dt}\Delta + 2$ and $e^{i\int_0^\pi \Im q dt}\Delta -2$ are double
 zeros. Thus every zero of $\tilde\Delta_a + 2$ and $\tilde\Delta_a -2$ is a double zero.
Applying Theorems \ref{perthm} and \ref{antiperthm} we have that $\tilde{Q}_1$ is both a.e. $\frac{\pi}{2}$-periodic and a.e.
 $\frac{\pi}{2}$-anti-periodic,
 thus $\tilde{Q}_1 = 0$ a.e.. So that $\tilde{Q} = \tilde{Q}_2$ a.e.. Thus equation (\ref{Q1Relation21}) shows that
$Q_1 = 0$ a.e. and $Q = r\sigma_0 + q\sigma_2$ a.e.. \qed

%%%%%%%%%%%%%%%%% bibliography %%%%%%%%%%%%%%%%%%%%%%

%%%%%%%%%%%%%%%%%%%%%%%%%%%%%%%%%%%%%%%%%%%%%%


\begin{thebibliography}{22}

\bibitem{ABar}
{\sc V.~A.~Ambarzumyan,}
{\em \"{U}ber eine Frage der Eigenwerttheorie},
{Z. Phys.} {\bf 53}, 690--695 (1912).	


\bibitem{BG}
{\sc G.~Borg,}
{\em Eine umkehrung der Sturm-Liouvillschen eigenwertaufgabe. bestimmung
der differentialgleichung durch die eigenwerte},
{Acta Math.} {\bf 78}, 1-96 (1946).	

%%%%%%%%%%%%%%%%% History of hermit syst %%%%%%%%%%%%%%%%%%%%%%

\bibitem{LaS}
{\sc  B.~M.~Levitan, I.~S.~Sargsjan},
{\em Sturm-Liouville and Dirac operators}, {\bf 59},
Kluwer Academic Publishers, (1991).

\bibitem{MlaMM}
{\sc  M.~Lesch, M.~Malamud},
{\em The inverse spectral problem for first order systems on the half line, in Operator Theory, Systems Theory
, and Related Topics: The Moshe Liv\u{s}ic Anniversary Volume}, D. Alpay and V. Vinnikov (eds.), Operatory Theory: Advances
and Applications, {\bf 117}, Birkh\"{a}user, Basel, p. 199--238 (2000).

\bibitem{LS1}
{\sc  A.~L.~Sakhnovich},
{\em Dirac type and canonical systems: spectral and Weyl-Titchmarch functions, direct and inverse problems}, Inverse Problems
{\bf 18}, 331--348 (2002).

\bibitem{LS2}
{\sc  A.~L.~Sakhnovich},
{\em Spectral Theory of Canonical Differential systems. Method of Operator Identities}, Operator Theory: Advances and
 Applications {\bf 107}, Birkh\"{a}user, Basel, (1999).

%%%%%%%%%%%%%%%%% History of hermit syst %%%%%%%%%%%%%%%%%%%%%%

%%%%%%%%%%%%%%%%% uniqueness %%%%%%%%%%%%%%%%%%%%%%

\bibitem{GKM}
{\sc  F.~Gesztesy, A.~Kiselev, K.~A.~Makarov},
{\em Uniquness results for matrix-valued Schr\"{o}dinger, Jacobi and Dirac-type operators}, 
Math. Machr. {\bf 239-240}, 103-145 (2002).

\bibitem{fGsC}
{\sc S.~Clark, F.~Gesztesy, H.~Holden, B.~M.~Levitan},
{\em Borg-Type Theorems for Matrix-Valued Schr\"{o}dinger Operators}, J. of Diff. Eqns. {\bf 167}, 181--210 (2000).

\bibitem{fGsC2}
{\sc S.~Clark, F.~Gesztesy, W.~Renger},
{\em Trace formulas and Borg-type theorems for matrix-valued Jacobi and Dirac finite difference operators}, J. of Diff. Eqns. {\bf 219}, 144--182 (2005).

\bibitem{fGmZ}
{\sc F.~Gesztesy, M.~Zinchenko},
{\em Borg-Type Theorem associated with orthogonal polynomials on the unit circle}, J. London. Math. Soc. (2) {\bf 74}, 757--777 (2006).

\bibitem{fGsC3}
{\sc S.~Clark, F.~Gesztesy},
{\em Weyl-Titchmarsh $M$-function asymptotics, local uniqueness results, trace formulas, and Borg-type Theorems for Dirac operators}
, J. London. Math. Soc. (2) {\bf 74}, 757--777 (2006).

\bibitem{fGaKkM}
{\sc F.~Gesztesy, A.~Kiselev, K.~A.~Makarov},
{\em Uniqueness results for matrix-valued Schr\"{o}dinger, Jacobi, and Dirac-type operators}
, Math. Nachr. {\bf 239-240}, 103--145 (2002).

%%%%%%%%%%%%%%%%% uniqueness %%%%%%%%%%%%%%%%%%%%%%

\bibitem{mK}
{\sc M.~Kriss},
{\em An $n$-dimensional Ambarzumian type theorem for Dirac Operators}, Inverse Problems {\bf 20}, 1593--1597 (2004).

\bibitem{cYxY}
{\sc C-F.~Yang, X-P~Yang},
{\em Some Ambarzumyan-type theorems for Dirac operators}, Inverse Problems {\bf 23}, 2565--2574 (2007).

\bibitem{fS1}
{\sc F.~Serier},
{\em Inverse spectral problems for singular Ablowitz-Kaup-Newell-Segur operators on $[0,1]$}
, Inverse Problems {\bf 22}, 1457--1484 (2006).

\bibitem{fS2}
{\sc F.~Serier},
{\em Inverse spectral problem for singular AKNS and Schr\"{o}dinger operators on }
, C. R. Acad. Sci. Paris, Ser I {\bf 340},  671--676 (2005).

\bibitem{HS}
{\sc  D.~B.~Hinton, J.~K.~Shaw},
{\em On Titchmarsh-Weyl $M(\lambda)$-functions for linear Hamiltonian systems}, 
J. Diff. Eq. {\bf 40}, 316--342 (1981).

%%%%%%%%%%%%%%%%%%%%%%%%%%%%%% dirac uniquness results %%%%%%%%%%%%%%%%%%%%%%%%%%%%%%%%%%%%%
\bibitem{mGtD1}
{\sc M.~G.~Gasymov, T.~T.~Dzabiev},
{\em Solution of the inverse problem by two spectra for the Dirac equation on a finite interval}, 
Akad. Nauk Azerbuidzan. SSR Dokl. {\bf 22}, 3--6 (1966).

\bibitem{mGtD2}
{\sc M.~G.~Gasymov, T.~T.~Dzabiev},
{\em Determination of the system of Dirac differential equations from two spectra}, 
Proc. of the SummerSchool in the Spectral Theory of Operators and the Theory of Group Representations  3--6 (1968).

\bibitem{mGtD3}
{\sc M.~G.~Gasymov, T.~T.~Dzabiev},
{\em The inverse problem for the Dirac system}, 
Dokl. Akad. Nauk SSSR {\bf 167}, 967--970 (1966).

\bibitem{bWat}
{\sc B.~A.~Watson},
{\em Inverse spectral problems for weighted Dirac systems}, 
Inverse Problems {\bf 15},  793--805 (1999).
%%%%%%%%%%%%%%%%%%%%%%%%%%%%%% dirac uniquness results %%%%%%%%%%%%%%%%%%%%%%%%%%%%%%%%%%%%%

%%%%%%%%%%%%%%%%%%%%%%%% AKNS %%%%%%%%%%%%%%%%%%%%%%%%%%%%%%%%%%%%%%%%%%%%%%%%%

\bibitem{lA}
{\sc L.~Amour},
{\em Inverse spectral theory for the AKNS system with separated boundary conditions}, 
Inverse Problems {\bf 9},  503--523 (1993).

\bibitem{lA2}
{\sc L.~Amour, J.-C.~Guillot},
{\em Isospectral sets for AKNS systems on the unit interval with generalised periodic boundary conditions}, 
Geom. Funct. Anal. {\bf 6}, 1--27 (1996).

\bibitem{fGKM}
{\sc F.~Gesztesy, A.~Kiselev, K.~A.~Marakov},
{\em Uniqueness results for matrix-valued Schr\"{o}dinger, Jacobi, and Dirac-type operators}, 
Math. Nachr., to appear (2001).


%%%%%%%%%%%%%%%%%%%%%%%% AKNS %%%%%%%%%%%%%%%%%%%%%%%%%%%%%%%%%%%%%%%%%%%%%%%%%

%%%%%%%%%%%%%%%%%%%%%%%% shabat-zakarov equation %%%%%%%%%%%%%%%%%%%%%%%%%%%%%%%%%%%%%%%%%%%%%%%%%

\bibitem{pBmV}
{\sc P.~Boonserm, M.~Visser},
{\em Reformulating the Schr\"{o}dinger equation as a Shabat-Zakharov system}, 
J. Math. Phys.,{\bf 51}, (2010).

\bibitem{mDdA}
{\sc M.~Desaix, D.~Anderson, L.~Helczynski, and M.~Lisak},
{\em Eigenvalues of the Zakharov-Shabat Scattering Problem for Real Symmetric Pulses}, 
Phys. Rev. Lett., {\bf 90}, (2006).

\bibitem{vsGgV}
{\sc V.~S.~Gerdjikov, G.~Vilasi, A.~B.~Yanovski},
 {\em The Inverse Scattering Problem for the Zakharov–Shabat System}, 
{\em Integrable Hamiltonian Hierarchies
Lecture Notes in Physics}, {\bf 748},  97--132 (2008)


%%%%%%%%%%%%%%%%%%%%%%%% shabat-zakarov equation %%%%%%%%%%%%%%%%%%%%%%%%%%%%%%%%%%%%%%%%%%%%%%%%%

%%%%%%%%%%%%%%%%%%%%%%%% Integrable systems %%%%%%%%%%%%%%%%%%%%%%%%%%%%%%%%%%%%%%%%%%%%%%%%%
\bibitem{nAyK}
{\sc N.~Asano, Y.~Kato}
{\em Algebraic and Spectral Methods for Nonlinear Wave equations},
Longman, New York, (1990).

\bibitem{IC}
{\sc I.~Cherednik}
{\em Basic Methods of Soliton Theory},
World Scientific, Singapore, (1996).

\bibitem{laD}
{\sc L.~A.~Dickey}
{\em Soliton Equations and Hamiltonian systems},
World Scientific, Singapore, (1991).

\bibitem{baD}
{\sc B.~A.~Dubrovin}
{\em Completely integrable Hamiltonian Systems associated with matrix operators and Abelian varieties},
Funct. Anal. Appl. {\bf 11 }, 265--277 (1977).
%%%%%%%%%%%%%%%%%%%%%%%% Integrable systems %%%%%%%%%%%%%%%%%%%%%%%%%%%%%%%%%%%%%%%%%%%%%%%%%

\bibitem{HH1}
{\sc H.~Hochstadt,}
On the determination of a Hill's equation from its spectrum,
{\em Archive for Rational Mechanics and Analysis}, {\bf 19}, (1965) 353--362.

\bibitem{HH2}
{\sc H.~Hochstadt,}
\newblock On a Hill's Equation with double Eigenvalues,
\newblock {\em Proceedings of the American Math. Soc.}, 65 (1977) 373--374.

\bibitem{HH3}
{\sc H.~Hochstadt,}
\newblock A direct and inverse problem for a Hill’s equation with double eigen-values, 
\newblock {\em J. Math. Anal. Appl.}, {\bf 66}, (1978) 507--513.

\bibitem{hormander}
{\sc L.~H\"ormander,}
{\em Lectures on nonlinear hyperbolic differential equations},
  Math\'ematiques \& Applications 26,
Springer Verlag, (1997).

\bibitem{FaU}
{\sc G.~Freiling, V.~A.~Yurko}
{\em Inverse Sturm-Liouville Problems and Applications},
  Nova Science Publishers, (2001).

\bibitem{CodaLev}
{\sc E.~A.~Coddington, N.~Levinson,}
{\em Theory of ordinary differential equations},
McGraw-Hill Publishing, (1955).

\bibitem{EBS}
{\sc M.~B,~Brown, M.~S.~P~Eastham, K.~M.~Schmidt,}
{\em Periodic Differential Operators},
Birkh\"{a}user Basel, (2013).




\end{thebibliography}
\end{document}